\documentclass[10pt]{amsart}

\usepackage{amsmath}
\usepackage{amsthm}
\usepackage{amssymb}

\usepackage{diagrams}

\theoremstyle{plain}
\newtheorem{dfn}[subsection]{Definition}
\newtheorem{thm}[subsection]{Theorem}
\newtheorem{prp}[subsection]{Proposition}
\newtheorem{cor}[subsection]{Corollary}
\newtheorem{lma}[subsection]{Lemma}

\theoremstyle{remark}
\newtheorem{rmk}[subsection]{Remark}

\newtheorem{exm}[subsection]{Example}
\newtheorem{exms}[subsection]{Examples}

\def\RR{\mathbb{R}}
\def\wRR{\widehat{\mathbb{R}}}
\def\SS{\mathbb{S}}

\def\inc{\hookrightarrow}
\def\Ob{\mathrm{Ob}}
\def\NN{\mathbb{N}}
\def\Iso{\mathrm{Iso}}
\def\op{\mathrm{op}}
\def\EE{\mathcal{E}}
\def\GG{\mathbb{G}}
\def\lrto{\longrightarrow}
\def\CC{\mathbb{C}}
\def\DD{\mathbb{D}}
\def\Cat{\mathrm{Cat}}
\def\al{\alpha}
\def\be{\beta}

\def\Aut{\mathrm{Aut}}

\def\eqv{\overset\sim\lrto}
\def\rto{\longrightarrow}
\def\onto{\twoheadrightarrow}
\def\Fin{Fin}
\def\Sets{Sets}
\def\OO{\mathcal{O}}
\def\card{\mathrm{card}}
\def\dim{\mathrm{dim}}
\def\sg{\sigma}
\def\Hom{\mathrm{Hom}}
\def\Sg{\Sigma}
\def\sk{\mathit{sk}}
\def\cosk{\mathit{cosk}}
\def\Arr{\mathrm{Arr}}

\def\ZZ{\mathbb{Z}}

\begin{document}
\title{On an extension of the notion of Reedy category}

\author{Clemens Berger and Ieke Moerdijk}

\date{September 17, 2008}

\subjclass{Primary 18G55, 55U35; Secondary 18G30, 20N99}
\keywords{Generalized Reedy category, Quillen model category, crossed group, dendroidal set}

\begin{abstract}We extend the classical notion of a Reedy category so as to allow non-trivial automorphisms. Our extension includes many important examples occuring in topology such as Segal's category $\Gamma$, or the total category of a crossed simplicial group such as Connes' cyclic category $\Lambda$. For any generalized Reedy category $\RR$ and any cofibrantly generated model category $\EE$, the functor category $\EE^\RR$ is shown to carry a canonical model structure of Reedy type.\end{abstract}

\maketitle

\section*{Introduction.}

A Reedy category is a category $\RR$ equipped with a structure which
makes it possible to prove that, for any Quillen \cite{Qu1} model
category $\EE$, the functor category $\EE^\RR$ inherits a model
structure, in which the cofibrations, weak equivalences and
fibrations can all three be described explicitly in terms of those
in $\EE$.  Prime examples of such Reedy categories are the simplex
category $\Delta$ and its dual $\Delta^\op$; the corresponding model
structure on cosimplicial spaces goes back to Bousfield and Kan
\cite{BK}, while the model structure on simplicial objects in an
arbitrary model category $\EE$ is described in an unpublished
manuscript by Reedy \cite{Ree}. The general result for an arbitrary
model category $\EE$ and Reedy category $\RR$ is by now a standard
and important tool in homotopy theory, well explained in several
textbooks, see e.g. \cite{GJ,Hi,Ho}.

As is well known, Reedy categories are skeletal, and moreover do not
permit non-trivial automorphisms. There are, however, important
cases in which it is possible to establish a Reedy-like model
structure on the functor category $\EE^\RR$  even though $\RR$ does
have non-trivial automorphisms. One example is the strict model
structure on $\Gamma$-spaces (space-valued presheaves on Segal's
\cite{Se} category $\Gamma$) established by Bousfield-Friedlander
\cite{BF}.  Another example is the case of cyclic spaces
(space-valued presheaves on Connes' category $\Lambda$, see
\cite{Co}). This paper grew out of a third example, namely the
category of dendroidal spaces \cite[Section 7]{MW} which carries a
Reedy-like model structure, although a dendroidal space is by
definition a presheaf on a category $\Omega$ of trees containing
many automorphisms. We expect this Reedy-like model structure on
dendroidal spaces (or a localization thereof) to be closely related
to a model structure on coloured topological operads, although the
precise relation remains to be worked out.

In this paper, we introduce the notion of a \emph{generalized Reedy
category}, and prove that for any such category $\RR$ and any
$\RR$-projective (e.g. cofibrantly generated) Quillen model category
$\EE$, the functor category $\EE^\RR$ inherits a model structure, in
which the cofibrations, weak equivalences and fibrations can again
be described explicitly in terms of those in $\EE$. Any classical
Reedy category is a generalized Reedy category in our sense; in
fact, a generalized Reedy category is equivalent to a classical one
if and only if it has no non-trivial automorphisms. Segal's category
$\Gamma$ (as well as its dual) and the cyclic category $\Lambda$ of
Connes are examples of generalized Reedy categories, as is any
(finite) group or groupoid.  The cyclic category is an example of
the total category associated to a crossed simplicial group
\cite{FL,Kr}; we will show that the total category of any
\emph{crossed group} on a classical Reedy category is a generalized
Reedy category. This method yields many interesting examples of
generalized Reedy categories with non-trivial automorphisms. In
particular, the category $\Omega$ mentioned above is of this type.
Other examples of generalized Reedy categories relevant in homotopy
theory are the \emph{orbit category} of a finite or compact Lie
group, and the total category associated to a \emph{complex of
groups}, see e.g. \cite{Ha}.

The results of this paper lead to several interesting questions. We
already mentioned the comparison between dendroidal spaces and
coloured topological operads, which we expect to be analogous to the
comparison between complete Segal spaces (a localization of the
Reedy model structure on simplicial spaces) and topologically
enriched categories --  see \cite{Bg, JT, Lu, Rez}. We expect the
Reedy model structure on spaces over a complex of groups to be
useful in describing the derived category of the corresponding
orbifold. A precise comparison would refine the weak homotopy
equivalence between (the classifying spaces of) the complex of
groups and the proper etale groupoid of the corresponding orbifold,
cf. \cite{Mo}.  Another topic to be explored further is the relation between various models for cyclic homology (see e.g. \cite{Qu2}) and the Reedy model structure on cyclic spaces given by applying our main theorem to Connes' category $\Lambda$. In this context, we note that it is known \cite{Ci} that a localization of this model structure is Quillen equivalent to the model structure \cite{DHK} on cyclic sets.

In a recent paper, Angeltveit \cite{Ang} studies Reedy categories
enriched in a monoidal model category, and obtains examples of such
from non-symmetric operads. We expect that a similar enrichment is
possible for our generalized Reedy categories, so that Angeltveit's
construction can be applied to symmetric operads as well. It would
also be of interest to extend the results of Barwick \cite{Ba} to
our context.\vspace{1ex}

To conclude this introduction, we describe the contents of the
different sections of this paper. In Section 1, we present our
notion of generalized Reedy category, state the main theorem on the
existence of a model structure (Theorem \ref{model}), and list some
of the main examples. In Section 2, we explain a general method for
constructing generalized Reedy categories out of classical ones by
means of crossed groups. Sections 3 and 4 contain some technical
preliminaries for the proof of the main theorem which will be given
in Section 5. In Section 6, we give a brief introduction into
skeleta and coskeleta for functor categories of the form $\EE^\RR$.
We then discuss a special class of dualizable generalized Reedy
categories $\RR$ for which the skeleta of set-valued presheaves on
$\RR$ have a simple, explicit description. In Section 7, we obtain a
refinement of the main theorem (Theorem \ref{monoidalmodel}) giving
sufficient conditions on $\RR$ and $\EE$ for the Reedy model
structure on $\EE^{\RR^\op}$ to be monoidal.

\vspace{1ex}

\emph{Acknowledgements:} The results of this paper were presented at the CRM in Barcelona in February 2008, in the context of the Program on Homotopy Theory and Higher Categories. The actual writing of this paper was done while the second author was visiting the University of Nice in May 2008. He would like to express his gratitude to the CRM and the University of Nice for their hospitality and support.\newpage

\section{Generalized Reedy categories.}

Recall that a subcategory $\SS$ of $\RR$ is called \emph{wide} if $\SS$ has the same objects as $\RR$. An example of a wide subcategory of $\RR$ is the \emph{maximal subgroupoid} $\Iso(\RR)$ of $\RR$.

\begin{dfn}\label{basic}A \emph{generalized Reedy structure} on a small category $\RR$ consists of wide subcategories $\RR^+,\RR^-$, and a degree-function $d:\Ob(\RR)\to\NN$ satisfying the following four axioms:\vspace{1ex}

\begin{enumerate}\item[(i)]non-invertible morphisms in $\,\RR^+$ (resp. $\RR^-$) raise (resp. lower) the degree; isomorphisms in $\RR$ preserve the degree;\item[(ii)]$\RR^+\cap\RR^- =\Iso(\RR)$;\item[(iii)]every morphism $f$ of $\,\RR$ factors as $f=gh$ with $g\in\RR^+$ and $h\in\RR^-$, and this factorization is unique up to isomorphism;\item[(iv)] If $\theta f=f$ for  $\,\theta\in\Iso(\RR)$ and $\,f\in\RR^-$, then $\theta$ is an identity.\end{enumerate}A generalized Reedy structure is \emph{dualizable} if in addition the following axiom holds:\begin{enumerate}\item[(iv)$'$]If $f\theta =f$ for  $\,\theta\in\Iso(\RR)$ and $\,f\in\RR^+$, then $\theta$ is an identity.\end{enumerate}

A (dualizable) generalized Reedy category is a small category equipped with a (dualizable) generalized Reedy structure.

A morphism of generalized Reedy categories $\RR\to\SS$ is a functor which takes $\,\RR^+$ (resp. $\RR^-$) to $\,\SS^+$ (resp. $\SS^-$) and which preserves the degree.\end{dfn}

\begin{rmk}The inclusion from left to right in axiom (ii) follows from axiom (i). Axiom (iv) says that automorphisms in $\RR$ consider morphisms of $\RR^-$ as epimorphisms. This last axiom implies that the isomorphism in (iii) is unique. The axioms (i)-(iii) are self-dual while axiom (iv) is dual to axiom (iv)$'$. A generalized Reedy category $\,\RR$ is thus dualizable if and only if $\,\RR^\op$ is also a generalized Reedy category. Most of the examples that we are aware of are dualizable. The asymmetry in the definition is related to the asymmetry of the projective model structure on objects with a group action, which enters in Theorem \ref{model}; cf. the proof of Lemma \ref{RLP2}.\end{rmk}


\begin{rmk}If $\RR$ is a generalized Reedy category, an equivalence of categories $\RR'\eqv\RR$ induces a generalized Reedy structure on $\RR'$. In this sense, the existence of a generalized Reedy structure is invariant under equivalence of categories.\end{rmk}

\begin{rmk}\label{strict}Recall that in the literature (cf. \cite{GJ, Hi, Ho, Ree}) a category $\RR$, equipped with $\RR^+$, $\RR^-$ and $d$ as above, is called a \emph{Reedy category} if it satisfies the following two axioms:\begin{enumerate}\item[(i)]non-identity morphisms in $\RR^+$ (resp. in $\RR^-$) raise (resp. lower) degree;\item[(ii)] every morphism in $\RR$ factors uniquely as a morphism in $\RR^-$ followed by one in $\RR^+$.\end{enumerate}Any such Reedy category is a dualizable generalized Reedy category in our sense. To emphasize the distinction with generalized Reedy categories we will refer to the classical ones as \emph{strict Reedy categories}. The notion of a strict Reedy category is not invariant under equivalence of categories. In fact, one checks that in a strict Reedy category every isomorphism is an identity.\footnote{Indeed, for an isomorphism $f$, let $f=gh$ and $hf^{-1}=g'h'$ be the unique factorizations. Then $id=ghf^{-1}=(gg')h'$, so $h'=id$ and $gg'=id$, whence $g=id$ and $g'=id$ since $g,g'\in\RR^+$. Thus $f=h\in\RR^-$. The same argument applied to $f^{-1}$ shows that $f$ preserves the degree, hence $f=id$.} A generalized Reedy category is equivalent to a strict one if and only if it has no non-trivial automorphisms, and is itself strict if and only if it is moreover skeletal.\end{rmk}

\begin{rmk}As for strict Reedy categories, all the results concerning a fixed generalized Reedy category $\RR$ go through if the degree-function takes values in an arbitrary well-ordered set. (However, with these more general degree-functions, the notion of a morphism of Reedy categories is more subtle to define).\end{rmk}

For a generalized Reedy category $\RR$, we introduce the following
notions, which are classical in the case of a strict Reedy category.
For each object $r$ of $\RR$, the category $\RR^+(r)$  has as
objects the \emph{non-invertible} morphisms in $\RR^+$ with codomain
$r$, and as morphisms from $u:s\to r$ to $u':s'\to r$ all $w:s\to
s'$ such that $u=u'w$ . Observe that axiom (iii) implies that
$w\in\RR^+$; moreover, the automorphism group $\Aut(r)$ acts on the
category $\RR^+(r)$  by composition. For each functor $X:\RR\to\EE$
 and each object $r$ of $\RR$, the \emph{$r$-th latching object}
$L_r(X)$ of $X$ is defined to be$$L_r(X)=\varinjlim_{s\to
r}X_s$$where the colimit is taken over the category $\RR^+(r)$. We
will always assume $\EE$ to be sufficiently cocomplete for this
colimit to exist (in many examples this colimit is finite). Note
that $\Aut(r)$ acts on $L_r(X)$.

Dually, for each object $r$ of $\RR$, the category $\RR^-(r)$  has as objects the \emph{non-invertible} morphisms in $\RR^-$ with domain $r$, and as morphisms from $u:r\to s$ to $u':r\to s'$ all $w:s\to s'$ such that $u'=wu$ . Observe that axiom (iii) implies that $w\in\RR^-$; moreover, the automorphism group $\Aut(r)$ acts on the category $\RR^-(r)$  by precomposition. For each object $X$ of $\EE^\RR$ and each object $r$ of $\RR$, the \emph{$r$-th matching object} $M_r(X)$ of $X$ is defined to be$$M_r(X)=\varprojlim_{r\to s}X_s$$where the limit is taken over the category $\RR^-(r)$. We will always assume $\EE$ to be sufficiently complete for this limit to exist (in many examples this limit is finite). Note that $\Aut(r)$ acts on $M_r(X)$.

Each object $X$ of the functor category $\EE^\RR$ defines for any object $r$ of $\RR$ natural $\Aut(r)$-equivariant maps $L_r(X)\to X_r\to M_r(X)$. For a map $f:X\to Y$ in $\EE^\RR$ these give rise to \emph{relative} latching, resp. matching maps\begin{gather*}X_r\cup_{L_r(X)}L_r(Y)\lrto Y_r,\text{\quad resp.\quad}X_r\lrto M_r(X)\times_{M_r(Y)}Y_r.\end{gather*}

Recall that for any group (or groupoid) $\Gamma$ and any
\emph{cofibrantly generated} model category $\EE$, the category
$\EE^\Gamma$ of objects of $\EE$ with right $\Gamma$-action carries
a projective model structure, in which weak equivalences and
fibrations are defined by forgetting the $\Gamma$-action. In
general, a Quillen model category $\EE$ will be called
\emph{$\RR$-projective}, if for each object $r$ of $\,\RR$, the
category $\EE^{\Aut(r)}$ admits a projective model structure. For
$\RR$-projective model categories $\EE$, we introduce the following
notions:\vspace{1ex}

A map $f:X\to Y$ in $\EE^\RR$ is called a\vspace{1ex}

-- \emph{Reedy cofibration} if for each $r$, the relative latching map $X_r\cup_{L_r(X)}L_r(Y)\to Y_r$ is a cofibration in $\EE^{\Aut(r)}$;

-- \emph{Reedy weak equivalence} if for each $r$, the induced map $f_r:X_r\to Y_r$ is a weak equivalence in $\EE^{\Aut(r)}$;

-- \emph{Reedy fibration} if for each $r$, the relative matching map $X_r\to M_r(X)\times_{M_r(Y)}Y_r$ is a fibration in $\EE^{\Aut(r)}$.\vspace{1ex}

\noindent Observe that the automorphism group $\Aut(r)$ really enters only in the definition of a Reedy cofibration, by definition of the model structure on $\EE^{\Aut(r)}$ just described.

\begin{thm}\label{model}Let $\RR$ be a generalized Reedy category and let $\EE$ be an $\RR$-projective Quillen model category in which the relevant limits and colimits exist (for instance, $\EE$ can be any cofibrantly generated model category). With the above classes of Reedy cofibrations, Reedy weak
equivalences and Reedy fibrations, the functor category $\EE^\RR$ is
a Quillen model category.\end{thm}

The proof will be supplied in Section 5. Notice that if $\RR=\RR^+$ then the constant functor $\EE\to\EE^\RR$ sends weak equivalences and fibrations in $\EE$ to Reedy weak equivalences and Reedy fibrations in $\EE^\RR$. Thus, we obtain the following corollary which is well known for strict Reedy categories.

\begin{cor}Let $\EE$ and $\RR$ be as in Theorem \ref{model}. If $\,\RR=\RR^+$ then $\varinjlim:\EE^\RR\to\EE$ is a left Quillen functor.\end{cor}

\begin{exms}\label{basicexamples}We give a list of first examples which, among other things, show that generalized Reedy categories occur naturally in several different contexts in homotopy theory. More examples are provided in Section 2.\vspace{1ex}

(a) For completeness, we mention (again) that any strict Reedy
category is a dualizable generalized Reedy category (cf. Remark
\ref{strict}). This applies in particular to standard examples of
Reedy categories such as the simplex category $\Delta$ and its dual,
as well as to $(\NN,<)$, $\leftarrow\cdot\rightarrow$,
$\cdot\rightrightarrows\cdot$ (relevant for homotopy colimits of
sequences, for homotopy pushouts and for homotopy coequalizers).
Other examples are Joyal's category of finite disks and its dual
$\Theta$ (cf. \cite{Jo, Be}).\vspace{1ex}

(b) Segal's category $\Gamma$ (cf. \cite{Se}) is a dualizable generalized Reedy category. In fact, $\Gamma^\op$ is equivalent to the category $\Fin_*$ of finite pointed sets, and one can take $\Fin_*^+$ to consist of monomorphisms and $\Fin_*^-$ of epimorphisms, while the degree-function is given by cardinality. If $\EE$ is the category of simplicial sets, the Reedy model structure on $\EE^{\Gamma^{\op}}$ given by Theorem \ref{model} was discussed in Bousfield-Friedlander \cite{BF} and referred to as the strict model structure on $\Gamma$-spaces. The simplicial circle $\Delta[1]/\partial\Delta[1]$, when viewed as a functor $\Delta\to\Gamma$ (cf. \cite{Se,Be}), is a morphism of generalized Reedy categories.\vspace{1ex}

(c) The category $\Fin$ of finite sets carries a dualizable
generalized Reedy structure, analogous to the pointed case. A
skeleton of $Fin$ is often denoted by $\Delta_{sym}$, and
$\EE^{\Delta_{sym}^\op}$ is referred to as the category of
\emph{symmetric simplicial objects} in $\EE$, cf. \cite{Ant, Ci}.
The inclusion $\Delta\inc\Delta_{sym}$ is a morphism of generalized
Reedy categories.\vspace{1ex}

(d) Any group(oid) is a generalized Reedy category.\vspace{1ex}

(e) \emph{Orbit categories}. The orbit category $\OO(G)$ of a \emph{finite group} $G$ has the subgroups of $G$ as objects, and the $G$-equivariant maps $G/H\to G/K$ as morphisms. This orbit category is a generalized Reedy category with $\OO(G)=\OO(G)^-$ and $d(H)=\card(G/H)$ (the index of $H$ in $G$). There is also a dual generalized Reedy structure on $\OO(G)$ with $\OO(G)=\OO(G)^+$ and $d(H)=\card(H)$. If $G$ is not finite, the first structure still makes sense for subgroups of finite index, the second one for finite subgroups. The orbit category $\OO(G)$ of a \emph{compact Lie group} $G$ is the category with closed subgroups of $G$ as objects and $G$-homotopy classes of $G$-maps $G/H\to G/K$ as morphisms from $H$ to $K$. This is again a generalized Reedy category with $\OO(G)=\OO(G)^+$. The degree of an object $H$ now takes values in $\NN\times\NN$ with the lexicographical ordering, and is defined by $d(H)=(\dim(H),\card(\pi_0H))$. Notice that this generalized Reedy structure is not in general dualizable like in the case of finite groups, because there may be infinite increasing sequences of closed subgroups, e.g. the subgroups $\ZZ/p^n\ZZ$ of the circle $S^1$.\vspace{1ex}

(f) \emph{Complexes of groups.} Let $X$ be a simplicial complex. Recall that a complex of groups $G$ over $X$ assigns to each simplex $\sg\in X$ a group $G_\sg$, to each inclusion $\sg\subseteq\tau$ an \emph{injective} group homomorphism $\phi_{\sg,\tau}:G_\tau\to G_\sg$, and to each sequence $\rho\subseteq\sg\subseteq\tau$ a specific element $g=g_{\rho,\sg,\tau}\in G_\rho$ such that the triangle\begin{diagram}[small,silent,UglyObsolete]G_\tau&\rTo&G_\sg\\&\rdTo&\dTo\\&&G_\rho\end{diagram}commutes up to conjugation by $g$, i.e. for each $x\in G_\tau$:$$g\phi_{\rho,\tau}(x)g^{-1}=\phi_{\rho,\sg}(\phi_{\sg,\tau}(x)).$$Moreover, for $\pi\subseteq\rho\subseteq\sg\subseteq\tau$, the following coherence condition should be satisfied:$$\phi_{\pi,\rho}(g_{\rho,\sg,\tau})g_{\pi,\rho,\tau}=g_{\pi,\rho,\sg}g_{\pi,\sg,\tau}.$$Such complexes of groups can be used to model orbifold structures on a triangulated space $|X|$, see \cite{Ha,Mo}. To each complex of groups $G$ over $X$ is associated a category $\Delta_X(G)$ whose objects are the simplices $\sg\in X$; if $\sg\subseteq\tau$ then morphisms $y:\sg\to\tau$ in $\Delta_X(G)$ are given by elements $y\in G_\sg$. Composition of $y:\sg\to\tau$ and $x:\rho\to\sg$ is defined to be $\phi_{\rho,\sg}(y)x:\rho\to\sg$. The coherence condition implies that this composition is associative. The category $\Delta_X(G)$ is a generalized Reedy category in which the degree of $\sg$ is the dimension of the simplex, and for which $\Delta_X(G)=\Delta_X(G)^+$. This example is a special case of Corollary \ref{closure2}.\end{exms}

The class of generalized Reedy categories is closed under arbitrary coproducts and under finite products. A more subtle closure property is the following:

\begin{prp}\label{closure1}Let $\,\SS\to\RR$ be a fibered category over $\RR$. Suppose that the base $\,\RR$ and each of the fibers $\,\SS_r$ are equipped with generalized Reedy structures. Assume furthermore that for each morphism $\al:r\to s$ in the base $\RR$,\begin{enumerate}\item[(i)]the base change $\al^*:\SS_s\to\SS_r$ preserves the degree;\item[(ii)]if $\al$ belongs to $\RR^+$ then $\al^*$ takes $\SS_s^+$ to $\,\SS_r^+$;\item[(iii)]if $\al$ belongs to $\RR^-$ then $\al^*$ has a left adjoint $\al_!$ which takes $\SS_r^-$ to $\,\SS_s^-$.\end{enumerate}Then $\SS$ can be equipped with a generalized Reedy structure such that the fiber inclusions $\,\SS_r\inc\SS$ and the projection $\,\SS\to\RR$ preserve the factorization systems.\end{prp}

\begin{proof}Consider a morphism $f:x\to y$ in $\SS$ over $\al:r\to s$ in $\RR$. Say $f\in\SS^+$ if $\al\in\RR^+$ and the unique morphism $x\to\al^*(y)$ in $\SS_r$ determined by a cartesian lift $\al^*(y)\to y$ of $\al$ lies in $\SS^+_r$. Say $f\in\SS^-$ if $\al\in\RR^-$ and the unique morphism $\al_!(x)\to y$ in $\SS_y$ determined by a cocartesian lift $x\to\al_!(x)$ of $\al$ lies in $\SS_s^-$. For $x\in\SS_r$, define the degree by $d_\SS(x)=d_\RR(r)+d_{\SS_r}(x)$. With these definitions, it is straightforward to verify that $\SS$ is a generalized Reedy category.\end{proof}

\begin{cor}\label{closure2}Let $\RR$ be a generalized Reedy category for which $\RR=\RR^+$, and let $\Phi:\RR^\op\to\Cat$ be a diagram of Reedy categories and morphisms of Reedy categories. Then the Grothendieck construction $\SS=\int_{\RR}\Phi$ is again a generalized Reedy category.\end{cor}

\section{Crossed groups.}

In this section, we introduce the notion of a \emph{crossed group} $G$ on a category $\RR$, and discuss the construction of the associated \emph{total category} $\RR G$. We will show that for any strict Reedy category $\RR$ and crossed $\RR$-group $G$, the total category $\RR G$ is a generalized Reedy category, which is no longer strict unless $G$ is trivial. Many of our examples of generalized Reedy categories are instances of this construction.

Crossed groups on the simplex category have been studied in the literature under the name \emph{skew-simplicial groups} (see Krasauskas \cite{Kr}), resp. \emph{crossed simplicial groups} (see Fiedorowicz-Loday \cite{FL}). Recently, Batanin-Markl \cite[2.2]{BM} considered \emph{crossed cosimplicial groups} which are crossed groups on the dual of the simplex category. Feigin-Tsygan already spelled out the axioms of a crossed group in \cite[A4.1-4]{FT}. Cisinski considers the more general concept of a \emph{thickening} in \cite[8.5.8]{Ci}.

\begin{dfn}\label{basic2}For any small category $\RR$, a \emph{crossed $\RR$-group} $\,G$ is a set-valued presheaf on $\RR$, together with, for each object $r$ of $\,\RR$,\begin{enumerate}\item[(i)]a \emph{group structure} on $G_r$,\item[(ii)]\emph{left $G_r$-actions} on the hom-sets $\Hom_\RR(s,r)$ with codomain $r$,\end{enumerate}such that the following identities hold for all $g,h\in G_r,\,\al:s\to r,\,\be:t\to s$,\begin{align}g_*(\al\circ\be)&=g_*(\al)\circ\al^*(g)_*(\be),\\g_*(1_r)&=1_r,\\
\al^*(g\cdot h)&=h_*(\al)^*(g)\cdot\al^*(h),\\ \al^*(e_r)&=e_s,\end{align}
where the presheaf action of $\al:s\to r$ is denoted by $\al^*:G_r\to G_s$ and the group action of $g\in G_r$ is denoted by $g_*:\Hom_\RR(s,r)\to\Hom_\RR(s,r)$. Moreover, for each object $r$, the identity of $\,r$ (resp. neutral element of $G_r$) is denoted by $\,1_r$ (resp. $e_r$).\end{dfn}

\begin{rmk}In what follows we shall make no difference in notation between composition in $\RR$ and composition in $G_r$, especially since both structures will agree in the total category $\RR G$. In addition to the four identities spelled out in Definition \ref{basic2}, the following four identities also hold in any $\RR$-crossed group $G$ (by the axioms for a presheaf, resp. group action):\begin{align}(\al\be)^*(g)&=\be^*\al^*(g),\\1_r^*(g)&=g,\\(gh)_*(\al)&=g_*h_*(\al),\\(e_r)_*(\al)&=\al.\end{align}
\end{rmk}

\subsection{The total category}For any small category $\RR$ and crossed $\RR$-group $G$, the total category $\RR G$ is the category with the same objects as $\RR$, and with morphisms $r\to s$ the pairs $(\al,g)$ where $\al:r\to s$ belongs to $\RR$, and $g\in G_r$. Composition of $(\al,g):s\to t$ and $(\be,h):r\to s$ is defined as
$$(\al,g)\circ(\be,h)=(\al\cdot g_*(\be),\be^*(g)\cdot h).$$
One easily checks that this composition is associative and has a two-sided unit $(1_r,e_r)$ for each object $r$ of $\RR G$.

\begin{rmk}In the special case where $G$ is a constant presheaf (i.e. $G=G_r$ for a fixed group $G$ and $\al^*(g)=g$ for all $g$ and all $\al$), the total category $\RR G$ reduces to the familiar Grothendieck construction for a diagram of categories on $G$.

In the special case where the left action of $G$ on $\RR$ is trivial (i.e. $g_*(\al)=\al$ for all $g$ and all $\al$), the crossed group is actually a presheaf of groups, and the total category $\RR G$ again reduces to a Grothendieck construction, this time for a diagram of groups on $\RR^\op$.\end{rmk}

Returning to the general case of a crossed $\RR$-group $G$, notice that we always have a canonical embedding $\RR\inc \RR G$ which sends $\al:r\to s$ to $(\al,e_r):r\to s$, and identifies $\RR$ with a wide subcategory of $\RR G$. Elements $g\in G_r$ of the crossed group may be identified with \emph{special automorphisms} $(1_r,g)$ in the total category $\RR G$, and every morphism $(\al,g)$ in $\RR G$ factors uniquely as a special automorphism $(1_r,g)$ followed by a morphism $(\al,e_r)$ in $\RR$. This \emph{unique factorization property} is characteristic for total categories of crossed groups as asserted by:

\begin{prp}\label{crossed1}Let $\RR\subseteq\SS$ be a wide subcategory and assume that there exist subgroups $G_s\subseteq\Aut_\SS(s)$ of special automorphisms such that each morphism in $\,\SS$ factors uniquely as a special automorphism followed by a morphism in $\,\RR$. Then the groups $G_s$ define a crossed $\,\RR$-group, and $\,\SS$ is isomorphic to $\RR G$ (under $\RR$).\end{prp}
\begin{proof}For any morphism $\al:r\to s$ of $\RR$ and special automorphism $g\in G_s$, the presheaf action of $\RR$ as well as the group action of $G$ are defined by factoring the composite $g\alpha:r\to s$ uniquely as in the hypothesis of the proposition, as \begin{diagram}[small]r&\rTo^\al&s\\\dTo^{\al^*(g)}&&\dTo_g\\r&\rTo_{g_*(\al)}&s.\end{diagram}With this explicit description, the proof of the identities of Definition \ref{basic2} and of the isomorphism $\SS\cong\RR G$ is a matter of (lengthy but) straightforward verification.\end{proof}

\begin{rmk}Fiedorowicz-Loday \cite{FL} take Proposition \ref{crossed1} for $\RR=\Delta$ as the definition of a crossed simplicial group (with contravariant instead of covariant group action), and state Definition \ref{basic2} of a crossed $\Delta$-group as a proposition.\end{rmk}

\begin{exm}\label{cyclic}The most prominent example of a crossed group is the \emph{simplicial circle} $C=\Delta[1]/\partial\Delta[1]$ whose total category $\Delta C$ is isomorphic to the \emph{cyclic category} $\Lambda$ of  Connes \cite{Co}. It is convenient to embed $C$ in a larger crossed $\Delta$-group $\Sg$, formed by the permutation groups $\Sg_{[n]}$ of the sets $[n]=\{0,1,\dots,n\}$. The crossed $\Delta$-group structure of $\Sg$ is defined as follows: given $\al:[m]\to[n]$ in $\Delta$ and $g:[n]\to[n]$ in $\Sg_{[n]}$, the map $\al^*(g):[m]\to[m]$ is the unique permutation which is \emph{order-preserving on the fibers} of $\al$, and for which $g_*(\al)=g\circ\al\circ\al^*(g)^{-1}:[m]\to[n]$ is order-preserving:\begin{diagram}[small][m]&\rTo^\al&[n]\\\dTo^{\al^*(g)}&&\dTo_g\\[m]&\rTo_{g_*(\al)}&[n].\end{diagram}Let $C_{[n]}\subset\Sg_{[n]}$ be the subgroup generated by the cycle $0\mapsto 1\mapsto\cdots\mapsto n\mapsto 0$. One checks that if $g\in C_{[n]}$ then $\al^*(g)\in C_{[m]}$ for each $\al:[m]\to[n]$ in $\Delta$, so that $C$ inherits a crossed $\Delta$-group structure. The total category $\Delta C$ is then isomorphic to the cyclic category $\Lambda$ of Connes \cite{Co}, and embeds in the total category $\Delta\Sg$. The latter has been described in detail by Feigin-Tsygan \cite[A10]{FT} and plays an important role in the general classification of crossed $\Delta$-groups, see \cite{FL,Kr}.\end{exm}


\begin{exm}\label{omega}One of the examples of a generalized Reedy category which motivated this paper is the category $\Omega$ of trees introduced by Moerdijk-Weiss in \cite{MW}. The objects of this category are finite trees with a distinguished output edge and a set of distinguished input edges, as common in the context of operads. Any such tree $T$ freely generates a symmetric coloured operad $\Omega(T)$ whose colour-set is the set $E(T)$ of edges of $T$; the morphisms $T\to T'$ in $\Omega$ are the maps of symmetric coloured operads $\Omega(T)\to\Omega(T')$.  For a more precise description, we refer to \cite{MW}. Here, it is enough to observe that any such morphism  $T\to T'$ induces a map $E(T)\to E(T')$ in a functorial way, and that this induced map completely determines the morphism. The category $\Omega$ carries a natural dualizable generalized Reedy structure, for which the degree is given by the number of vertices in the tree, while a morphism belongs to $\Omega^+$ (resp. $\Omega^-$) when it induces an injection (resp. surjection) between the sets of edges.

For such a tree $T$, one can consider the set of planar structures $p$ on $T$. Since every tree in $\Omega$ carries at least one planar structure, the category $\Omega$ is equivalent to the category $\Omega'$ whose objects are planar trees $(T,p)$, and whose morphisms $(T,p)\to(T',p')$ are the morphisms $T\to T'$ in $\Omega$. For every such morphism, one can pull back the planar structure $p'$ on $T'$ to one on $T$, and call the morphism planar if this pulled back structure coincides with $p$. The planar morphisms form a wide subcategory of $\Omega'$, denoted $\Omega_{planar}$; in this latter category, every automorphism is trivial, and  $\Omega_{planar}$ is equivalent to a strict Reedy category. Every morphism in  $\Omega'$ factors uniquely as an automorphism followed by a planar map. This shows by Proposition \ref{crossed1} that the category $\Omega$ is equivalent to the total category of a crossed group on $\Omega_{planar}$.

The embedding $i:\Delta\inc\Omega$ (cf. \cite{MW}) is a morphism of
generalized Reedy categories, and Theorem \ref{model} gives a Reedy
model structure on \emph{dendroidal spaces}, which is compatible
with the Reedy model structure on simplicial spaces. At the end of
Section 7 (cf. Exampe \ref{monoidalEil}(iii)), we will show that the
model structure on dendroidal spaces is \emph{monoidal} (in the
sense of Hovey \cite{Ho}) with respect to the \emph{Boardman-Vogt
tensor product} on dendroidal spaces (cf.
\cite[appendix]{MW}).\end{exm}

Consider a crossed $\RR$-group $G$, and suppose that $\RR$ carries a generalized Reedy structure. We will say that the crossed $\RR$-group is \emph{compatible with the generalized Reedy structure} if the following two conditions hold:\vspace{1ex}

\begin{enumerate}\item[(i)]the $G$-action respects $\RR^+$ and $\RR^-$ (i.e. if $\al:r\to s$ belongs to $\RR^\pm$ and $g\in G_s$ then $g_*(\al):r\to s$ belongs to $\RR^\pm$);\item[(ii)]if $\al:r\to s$ belongs to $\RR^-$ and $g\in G_s$ is such that $\al^*(g)=e_r$ and $g_*(\al)=\al$, then $g=e_s$.\end{enumerate}\vspace{1ex}

\begin{rmk}Observe that condition (i) is in particular satisfied if any morphism in $\RR$, which in $\RR G$ is isomorphic to a morphism in $\RR^\pm$, already belongs to $\RR^\pm$. Condition (ii) is equivalent to the condition that $\RR^-$ fulfills axiom (iv) of Definition \ref{basic} with respect to special automorphisms of $\RR G$, cf. the proof of Proposition \ref{crossed1}.

Because in the simplex category $\Delta$ the morphisms of
$\,\Delta^+$ (resp. of $\,\Delta^-$) are the monomorphisms (resp.
split epimorphisms) of $\Delta$, \emph{any} crossed $\Delta$-group
is compatible with the Reedy structure of $\Delta$. The same
property holds for crossed groups on $\Omega_{planar}$, cf. Example
\ref{omega}, and in general for crossed groups on strict
EZ-categories, cf. Definition \ref{Eilenberg}.\end{rmk}

\begin{prp}\label{crossed2}Let $\RR$ be a strict Reedy category, and let $G$ be a compatible crossed $\,\RR$-group. Then there is a unique dualizable generalized Reedy structure on $\RR G$ for which the embedding $\RR\inc\RR G$ is a morphism of generalized Reedy categories.\end{prp}

\begin{proof}Necessarily, $(\RR G)^\pm$ consists of those morphisms $(\al,g)$ for which $\al\in\RR^\pm$. Because of compatibility condition (i), $(\RR G)^\pm$ is closed under composition. It is now straightforward to verify that this defines a generalized Reedy structure on $\RR G$. In particular, axiom (iv) follows from compatibility condition (ii) and the fact that all automorphisms of $\RR G$ are special since $\RR$ is a strict Reedy category. The dual axiom (iv)$'$ holds automatically.\end{proof}

\section{Kan extensions and the projection formula.}

In this section we recall some basic facts about Kan extensions for diagram categories. Let $\phi:\DD\lrto\CC$ be a functor between small categories, and write $\phi^*:\EE^\CC\lrto\EE^\DD$ for precomposition with $\phi$. The left and right adjoints of $\phi^*$ are usually called \emph{left} and \emph{right Kan extension} along $\phi$.

If $\EE$ is sufficiently cocomplete, the left Kan extension $\phi_!:\EE^\DD\lrto\EE^\CC$ can be computed pointwise by$$\phi_!(X)_c=\varinjlim_{\phi/c}X\circ\pi_c$$where $\phi/c$ is the comma category with objects $(d,\,u:\phi(d)\to c)$ and morphisms $(d,u)\to(d',u')$ given by $f:d\to d'$ in $\DD$ such that $u'\circ\phi(d)=u$. The functor $\pi_c:\phi/c\lrto\DD$ is defined by $(d,u)\mapsto d$. We will often informally write$$\phi_!(X)_c=\varinjlim_{\phi(d)\to c}X_d.$$
The formula for left Kan extension simplifies if the functor $\phi:\DD\to\CC$ is cofibered. Recall (cf. \cite{Bo}) that for a given functor $\phi:\DD\to\CC$, a morphism $f:d\to d'$ in $\DD$ is \emph{cocartesian} if for any $g:d\to d''$ such that $\phi(g)=h\phi(f)$, there is a unique $k:d'\to d''$ such that $g=kf$ and $\phi(k)=h$. The functor $\phi$ is called \emph{cofibered}, if morphisms in $\CC$ have cocartesian lifts in $\DD$, and if moreover cocartesian morphisms in $\DD$ are closed under composition. If $\phi$ is cofibered, then for any object $c$ of $\,\CC$, the embedding of the fiber $\phi^{-1}(c)$ into the comma category $\phi/c$ (given on objects by $d\mapsto (d,1_{\phi(c)})$) has a left adjoint, so $\phi^{-1}(c)$ is cofinal in $\phi/c$, and hence$$\phi_!(X)_c=\varinjlim_{\phi^{-1}(c)}X$$is the colimit over the fiber. This implies that for any pullback diagram of categories
\begin{diagram}[small]\DD'&\rTo^\be&\DD\\\dTo^{\psi}&&\dTo_{\phi}\\\CC'&\rTo^\al&\CC\end{diagram}with $\phi$ (and hence $\psi$) cofibered,  the natural transformation of functors $$\psi_!\be^*\rto\al^*\phi_!$$ is an \emph{isomorphism}. This is often called the \emph{projection formula}, and will be applied below in the special case where $\DD=\int_\CC F$ is the Grothendieck construction of a covariant diagram $F:\CC\to\Cat$.\vspace{1ex}

Dually, if $\EE$ is sufficiently complete, the right Kan extension $\phi_*:\EE^\DD\to\EE^\CC$ can be computed pointwise by$$\phi_*(X)_c=\varprojlim_{c\to\phi(d)}X_d$$and this formula simplifies for fibered functors $\phi:\DD\to\CC$. Recall that a functor $\phi$ is called \emph{fibered}, if morphisms in $\CC$ have \emph{cartesian} lifts in $\DD$, and if moreover cartesian morphisms in $\DD$ are closed under composition. If $\phi$ is fibered, then for any object $c$ of $\,\CC$, the embedding of the fiber $\phi^{-1}(c)$ into the comma category $c/\phi$ (given on objects by $d\mapsto (1_{\phi(c)},d)$) has a right adjoint, so $\phi^{-1}(c)$ is final in $c/\phi$, and hence$$\phi_*(X)_c=\varprojlim_{\phi^{-1}(c)}X$$is the limit over the fiber.
This implies that for any pullback diagram of categories\begin{diagram}[small]\DD'&\rTo^\be&\DD\\\dTo^{\psi}&&\dTo_{\phi}\\\CC'&\rTo^\al&\CC\end{diagram}with $\phi$ (and hence $\psi$) \emph{fibered}, the natural transformation of functors $$\al^*\phi_*\rto\psi_*\be^*$$ is an \emph{isomorphism}. This \emph{dual projection formula} will be applied below in the special case where $\DD=\int_{\CC} F$ is the Grothendieck construction of a contravariant diagram $F:\CC^\op\to\Cat$.\vspace{1ex}

\section{Latching and matching objects.}

In this section we give an alternative, more global definition of latching and matching objects. Throughout, we consider a fixed generalized Reedy category  $\RR$ with wide subcategories $\RR^\pm$ and degree-function $d$ as in Definition \ref{basic}, and assume that $\EE$ is a sufficiently bicomplete category.

\subsection{The groupoids of objects of fixed degree.}\label{groupoiddef}For each natural number $n$, the full \emph{subcategory} of $\RR$ of objects of degree $\leq n$ will be denoted $\RR_{\leq n}$, the full \emph{subgroupoid} of $\Iso(\RR)$ spanned by the objects of degree $n$ will be denoted $\GG_n(\RR)$, and the \emph{discrete} category of objects of $\RR$ of degree $n$ will be denoted $\RR_n$.

\subsection{Overcategories.}For each natural number $n$, the category $\RR^+((n))$ has as objects the non-invertible morphisms $u:s\to r$ in $\RR^+$ such that $d(r)=n$, and as morphisms from $u$ to $u'$ the commutative squares\begin{diagram}[small]s&\rTo^f&s'\\\dTo^u&&\dTo_{u'}\\r&\rTo^g&r'\end{diagram}such that $f\in\RR^+$ and $g\in\GG_n(\RR)$.

The wide subcategory $\RR^+(n)$ of $\RR^+((n))$ contains those morphisms for which $g$ is an identity. The category $\RR^+(r)$ of Section 1 may thus be identified with the full subcategory of $\RR^+(n)$ spanned by the objects with codomain $r$. Notice that $$\RR^+(n)=\coprod_{d(r)=n}\RR^+(r).$$

The categories introduced so far assemble into the following commutative diagram:\begin{diagram}[small]\RR&\lTo^{d_n}&\RR^+((n))&\rTo^{c_n}&\GG_n(\RR)&\rTo^{j_n}&\RR\\&&\uTo^{k_n}&&\uTo_{i_n}&&\\&&\RR^+(n)&\rTo^{b_n}&\RR_n\end{diagram}where $d_n$ denotes the domain-functor, $b_n$ and $c_n$ denote codomain-functors, and $i_n$, $j_n$ and $k_n$ are inclusion-functors. Note that $c_n$ is cofibered, i.e.$$\RR^+((n))\cong\int_{\GG_n(\RR)}\RR^+(-),$$and that the square is a pullback. In particular, the projection formula yields $$i_n^*(c_n)_!\cong(b_n)_!k_n^*.$$

\subsection{Latching objects.}\label{latchingdef}The definition of the \emph{latching object} $L_n(X)$ for an object $X$ of $\EE^\RR$ now takes the following form:\begin{align*}L_n(X)&=(c_n)_!d_n^*(X)\in\EE^{\GG_n(\RR)}.\end{align*}We write $X_n=j_n^*(X)=X_{|\GG_n(\RR)}$, so that we get in each degree $n$ a \emph{latching map}\begin{align*}L_n(X)&\lrto X_n.\end{align*} Note that, since $c_n$ is cofibered, we have more concretely:\begin{align*}L_n(X)_r&=\varinjlim_{s\to r}X_s,\end{align*}where the colimit is taken over the category $\RR^+(r)$ as in Section 1. Accordingly, we will often simplify notation and write $L_r(X)$ for $L_n(X)_r$.\vspace{1ex}

Observe that a morphism $\phi:\SS\to\RR$ of generalized Reedy categories induces for $k\in\NN$ and $X\in\EE^\RR$ a natural map$$L_k(\phi^*(X))\rto\phi_k^*(L_k(X))$$where $\phi^*:\EE^\RR\to\EE^\SS$ and $\phi_k^*:\EE^{\GG_k(\RR)}\to\EE^{\GG_k(\SS)}$ are induced by $\phi$.

\begin{lma}\label{latching}Let $\phi:\SS\to\RR$ be a morphism of generalized Reedy categories. Suppose that the induced square\begin{diagram}[small]\SS^+((k))&\rTo&\GG_k(\SS)\\\dTo^{\phi^+_k}&&\dTo_{\phi_k}\\\RR^+((k))&\rTo&\GG_k(\RR)\end{diagram}is a pullback. Then, for each object $X$ of $\,\EE^\RR$, the natural comparison map of latching objects $L_k(\phi^*(X))\to\phi_k^*(L_k(X))$ is an isomorphism.\vspace{1ex}

The pullback hypothesis holds in particular in the following two cases:\begin{enumerate}\item[(i)]$\,\SS=\RR^+(n)$ and $\phi=d_nk_n:\RR^+(n)\to\RR$ is the domain functor;\item[(ii)]$\,\SS=\RR_{\leq n}$ and $\phi:\RR_{\leq n}\to\RR$ is the canonical embedding.\end{enumerate}\end{lma}

\begin{proof}The pullback square is part of the commutative diagram\begin{diagram}[small]\SS&\lTo^{\bar{d}_k}&\SS^+((k))&\rTo^{\bar{c}_k}&\GG_k(\SS)\\\dTo^\phi&&\dTo^{\phi^+_k}&&\dTo_{\phi_k}\\\RR&\lTo^{d_k}&\RR^+((k))&\rTo^{c_k}&\GG_k(\RR)\end{diagram}whose rows enter into the definition of the latching objects. Together with the projection formula, this yields canonical isomorphisms:\begin{align*}L_k(\phi^*(X))&=\bar{c}_{k!}\bar{d}_k^*\phi^*(X)\\&=\bar{c}_{k!}(\phi^+_k)^*d_k^*(X)\\&\cong\phi_k^*c_{k!}d^*_k(X)\\&=\phi_k^*(L_k(X)).\end{align*}

If $\,\SS=\RR^+(n)$ then $\,\SS^+((k))$ has as objects the composable pairs $\,t\to s\to r$ of non-invertible, degree-raising maps such that $d(r)=n$ and $d(s)=k$, and as morphisms those transformations of diagrams which are the identity on the last object, an isomorphism on the intermediate object, and degree-raising on the first object; this category coincides with the fiber product of $\phi_k:\GG_k(\SS)\to\GG_k(\RR)$ and $c_k:\RR^+((k))\to\GG_k(\RR)$.

If $\SS=\RR_{\leq n}$ the pullback hypothesis follows from the fact that an object of $\RR^+((k))$ belongs to $\,\SS^+((k))$ if and only if its codomain is of degree $\leq n$.\end{proof}

\subsection{Undercategories.}The category $\RR^-((n))$ has as objects the non-invertible morphisms $u:r\to s$ in $\RR^-$ such that $d(r)=n$, and as morphisms from $u$ to $u'$ the commutative squares\begin{diagram}[small]r&\rTo^g&r'\\\dTo^u&&\dTo_{u'}\\s&\rTo^f&s'\end{diagram}such that $f\in\RR^-$ and $g\in\GG_n(\RR)$.

The wide subcategory $\RR^-(n)$ of $\RR^-((n))$ contains those morphisms for which $g$ is an identity. The category $\RR^-(r)$ of Section 1 may then be identified with the full subcategory of $\RR^-(n)$ spanned by the objects with domain $r$. Notice that $$\RR^-(n)=\coprod_{d(r)=n}\RR^-(r).$$

The categories introduced so far assemble into the following commutative diagram:\begin{diagram}[small]\RR&\lTo^{\gamma_n}&\RR^-((n))&\rTo^{\delta_n}&\GG_n(\RR)&\rTo^{j_n}&\RR\\&&\uTo^{\kappa_n}&&\uTo_{i_n}&&\\&&\RR^-(n)&\rTo^{\be_n}&\RR_n\end{diagram}where $\gamma_n$ denotes the codomain-functor, $\be_n$ and $\delta_n$ denote domain-functors, and $i_n$, $j_n$ and $\kappa_n$ are inclusion-functors.

Note that $\delta_n$ is fibered, i.e. $$\RR^-((n))\cong\int_{\GG_n(\RR)}\RR^-(-),$$and that the square is a pullback. In particular, the dual projection formula yields $$i_n^*(\delta_n)_*\cong(\be_n)_*\kappa_n^*.$$

\subsection{Matching objects.}\label{matchingdef}The definition of the \emph{matching object} $M_n(X)$ of an object $X$ of $\EE^\RR$ now takes the following form:\begin{align*}M_n(X)=(\delta_n)_*\gamma_n^*(X)\in\EE^{\GG_n(\RR)}.\end{align*}We write $X_n=j_n^*(X)=X_{|\GG_n(\RR)}$, so that we get in each degree $n$ a \emph{matching map}\begin{align*}X_n&\lrto M_n(X).\end{align*} Note that, since $\delta_n$ is fibered, we have more concretely:\begin{align*}M_n(X)_r&=\varprojlim_{r\to s}X_s,\end{align*}where the limit is taken over the category $\RR^-(r)$ as in Section 1. Accordingly, we will often simplify notation and write $M_r(X)$ for $M_n(X)_r$.

\begin{lma}\label{matching}Let $\phi:\SS\to\RR$ be a morphism of generalized Reedy categories. Suppose that the induced square\begin{diagram}[small]\SS^-((k))&\rTo&\GG_k(\SS)\\\dTo^{\phi^-_k}&&\dTo_{\phi_k}\\\RR^-((k))&\rTo&\GG_k(\RR)\end{diagram}is a pullback. Then, for each object $X$ of $\EE^\RR$, the natural comparison map of matching objects $\phi_k^*(M_k(X))\to M_k(\phi^*(X))$ is an isomorphism.\vspace{1ex}

The pullback hypothesis holds in particular in the following two cases:\begin{enumerate}\item[(i)]$\,\SS=\RR^-(n)$ and $\phi=\gamma_n\kappa_n:\RR^-(n)\to\RR$ is the codomain functor;\item[(ii)]$\,\SS=\RR_{\leq n}$ and $\phi:\RR_{\leq n}\to\RR$ is the canonical embedding.\end{enumerate}\end{lma}

\begin{proof}Dual to the proof of Lemma \ref{latching}.\end{proof}

\section{The Reedy model structure.}

We can reformulate the definition of the classes of maps in Section 1 as follows:

\begin{lma}A map $X\to Y$ in $\EE^\RR$ is a Reedy cofibration (resp. a Reedy weak equivalence, resp. a Reedy fibration) if and only if, for each natural number $n$, the map $X_n\cup_{L_n(X)}L_n(Y)\to Y_n$ (resp. $X_n\to Y_n$, resp. $X_n\to M_n(X)\times_{M_n(Y)}Y_n)$ is a cofibration (resp. a weak equivalence, resp. a fibration) in $\EE^{\GG_n(\RR)}$.\end{lma}

\begin{proof}This just follows from the equivalence of categories$$\EE^{\GG_n(\RR)}\eqv\prod_r\EE^{\Aut(r)}$$where $r$ runs through a set of representatives for the connected components of the groupoid $\GG_n(\RR)$.\end{proof}

A Reedy (co)fibration which is also a Reedy weak equivalence will be referred to as a \emph{trivial Reedy (co)fibration}. The following lemmas are preparatory for the proof of Theorem \ref{model}.

\begin{lma}\label{LLP1}Let $f:A\to B$ be a trivial Reedy cofibration; suppose that, for each $n$, the induced map $L_n(f):L_n(A)\to L_n(B)$ is a pointwise trivial cofibration (i.e. $L_n(f)_r$ is a trivial cofibration in $\EE$ for each object $r$ of $\,\RR$). Then $f:A\to B$ has the left lifting property with respect to Reedy fibrations.\end{lma}

\begin{proof}Consider a commutative square in $\EE^\RR$\begin{diagram}[small]A&\rTo^\al&Y\\\dTo^f&&\dTo_g\\B&\rTo^\be&X\end{diagram}where $f$ is a trivial Reedy cofibration and $g$ is a Reedy fibration, and furthermore $L_n(f):L_n(A)\to L_n(B)$ is a pointwise trivial cofibration for all $n$. We construct a diagonal filler $\gamma:B\to Y$ by constructing inductively a filler $\gamma_{\leq n}:B_{\leq n}\to Y_{\leq n}$ on the full subcategory $\RR_{\leq n}$ of objects of $\RR$ of degree $\leq n$. For $n=0$, we get a diagonal filler $\gamma_0:B_0\to Y_0$ in\begin{diagram}[small]A_0&\rTo^{\al_0}&Y_0\\\dTo^{f_0}&&\dTo_{g_0}\\B_0&\rTo^{\be_0}&X_0\end{diagram}since $\RR_{\leq 0}$ is the groupoid $\GG_0(\RR)$, and $L_0(A)=0$, $M_0(X)=1$, so that by hypothesis $f_0$ is a trivial cofibration in $\EE^{\GG_0(\RR)}$ and $g_0$ is a fibration in $\EE^{\GG_0(\RR)}$.

Assume by induction that a filler $\gamma_{\leq n-1}:B_{\leq n-1}\to Y_{\leq n-1}$ has been found for \begin{diagram}[small]A_{\leq n-1}&\rTo^{\al_{\leq n-1}}&Y_{\leq n-1}\\\dTo^{f_{\leq n-1}}&&\dTo_{g_{\leq n-1}}\\B_{\leq n-1}&\rTo^{\be_{\leq n-1}}&X_{\leq n-1}.\end{diagram}This yields composite maps\begin{gather*}L_n(B)\lrto L_n(Y)\lrto Y_n\text{ and }B_n\lrto M_n(B)\lrto M_n(Y)\end{gather*}as well as a commutative square\begin{diagram}[small]A_n\cup_{L_n(A)}L_n(B)&\rTo&Y_n\\\dTo^{v_n}&&\dTo_{w_n}\\B_n&\rTo&X_n\times_{M_n(X)}M_n(Y).\end{diagram}A $\GG_n(\RR)$-equivariant filler is exactly what is needed in order to complete the inductive step. To see that such a filler exists, note that by hypothesis $v_n$ is a cofibration and $w_n$ is a fibration in $\EE^{\GG_n(\RR)}$. It is thus enough to check that $v_n$ is a weak equivalence. For this, consider the following diagram in which the square is a pushout: \begin{diagram}[small,UglyObsolete,silent]L_n(A)_r&\rTo&A_r&\rTo^{f_r\quad}&B_r\\\dTo^{L_n(f)_r}&&\dTo&\ruTo_{(v_n)_r}&\\L_n(B)_r&\rTo&(A\cup_{L_n(A)}L_n(B))_r.&&\end{diagram}Since, by hypothesis, $L_n(f)_r$ is a trivial cofibration in $\EE$, and $f_r$ is a weak equivalence, $v_n$ is a weak equivalence as required.\end{proof}

\begin{lma}\label{LLP2}Let $f:A\to B$ be a Reedy cofibration such that $f_r:A_r\to B_r$ is a weak equivalence for all objects $r$ of $\RR$ of degree $<n$. Then, the induced map $L_n(f):L_n(A)\to L_n(B)$ is a pointwise trivial cofibration.\end{lma}

\begin{proof}For $n=0$, there is nothing to prove; therefore, we can assume inductively that $L_k(f):L_k(A)\to L_k(B)$ is a pointwise trivial cofibration for $k<n$. We want to show that $i^*_nL_n(f)$ is a trivial cofibration in $\EE^{\RR_n}$. To this end, we have to find a filler for any commutative square\begin{diagram}[small]i_n^*L_n(A)&\rTo&Y\\\dTo^{i_n^*L_n(f)}&&\dTo_g\\i_n^*L_n(B)&\rTo&X\end{diagram}in $\EE^{\RR_n}$ in which $g:Y\to X$ is a fibration. Since $i^*_nL_n=i^*_n(c_n)_!d_n^*=(b_n)_!k_n^*d_n^*$, a filler for the former square is the same as a filler for the following square in $\EE^{\RR^+(n)}$:\begin{diagram}[small]k_n^*d^*_n(A)&\rTo&b_n^*(Y)\\\dTo^{k_n^*d^*_n(f)}&&\dTo_{b_n^*(g)}\\k_n^*d^*_n(B)&\rTo&b_n^*(X).\end{diagram}In order to finish the proof, we shall apply Lemma \ref{LLP1} to this square. The category $\SS=\RR^+(n)$ is a generalized Reedy category for which $\SS=\SS^+$. In particular, Reedy fibrations are the same as pointwise fibrations, so $b_n^*(g)$ is a Reedy fibration. Moreover, $k^*_nd_n^*(f)$ is a Reedy weak equivalence in $\EE^\SS$, since the objects of $\SS$ have degree $<n$. It remains to be shown that $k^*_nd_n^*(f)$ is a Reedy cofibration whose induced maps on latching objects of degree $<n$ are pointwise trivial cofibrations.

Write $\phi=d_nk_n$. By Lemma \ref{latching}, the functor $\phi_k^*:\EE^{\GG_k(\RR)}\to\EE^{\GG_k(\SS)}$ induces a canonical isomorphism $L_k(\phi^*(A))\cong\phi_k^*(L_k(A))$. Therefore, the relative latching map $\phi^*(A)\cup_{L_k(\phi^*(A))}L_k(\phi^*(B))\to\phi^*(B)$ may be identified with $\phi_k^*$ of the relative latching map $A_k\cup_{L_k(A)}L_k(B)\to B_k$. Observe that $\phi_k:\GG_k(\SS)\to\GG_k(\RR)$ is a faithful functor between groupoids, so $\phi_k^*$ preserves cofibrations, thus $k_n^*d_n^*(f)$ is a Reedy cofibration. Moreover, $L_k(\phi^*(A))\to L_k(\phi^*(B))$ is a pointwise trivial cofibration for $k<n$, since $L_k(A)\to L_k(B)$ is so by induction hypothesis.\end{proof}

\begin{lma}\label{RLP1}Let $g:Y\to X$ be a trivial Reedy fibration; suppose that for each $n$, the induced map $M_n(g):M_n(Y)\to M_n(X)$ is a (pointwise) trivial fibration. Then $g:Y\to X$ has the right lifting property with respect to Reedy cofibrations.\end{lma}

\begin{proof}Dual to the proof of Lemma \ref{LLP1}.\end{proof}

\begin{lma}\label{RLP2}Let $g:Y\to X$ be a Reedy fibration such that $g_r:Y_r\to X_r$ is a weak equivalence for all objects $r$ of $\RR$ of degree $<n$. Then, the induced map $M_n(g):M_n(Y)\to M_n(X)$ is a (pointwise) trivial fibration.\end{lma}

\begin{proof}For $n=0$, there is nothing to prove; therefore, we can assume inductively that $M_k(g):M_k(Y)\to M_k(X)$ is a trivial fibration for $k<n$. We want to show that $i^*_nM_n(g)$ is a trivial fibration in $\EE^{\RR_n}$. To this end, we have to find a filler for any commutative square\begin{diagram}[small]A&\rTo&i_n^*M_n(Y)\\\dTo^f&&\dTo_{i_n^*M_n(g)}\\B&\rTo&i_n^*M_n(X)\end{diagram}in $\EE^{\RR_n}$ in which $f:A\to B$ is a cofibration. Since $i^*_nM_n=i^*_n(\delta_n)_*\gamma_n^*=(\beta_n)_*\kappa_n^*\gamma_n^*$, a filler for the former square is the same as a filler for the following square in $\EE^{\RR^-(n)}$:\begin{diagram}[small]\beta_n^*(A)&\rTo&\kappa_n^*\gamma_n^*(Y)\\\dTo^{\beta_n^*(f)}&&\dTo_{\kappa_n^*\gamma_n^*(g)}\\\beta^*_n(B)&\rTo&\kappa_n^*\gamma_n^*(X).\end{diagram}In order to finish the proof, we shall apply Lemma \ref{RLP1} to this square. The category $\SS=\RR^-(n)$ is a generalized Reedy category for which $\SS=\SS^-$; notice that $\SS$ is has no non-trivial automorphisms in virtue of axiom (iv) of Definition \ref{basic}; in other words, $\SS$ is equivalent to a strict Reedy category. Therefore, Reedy cofibrations are the same as pointwise cofibrations, so $\beta_n^*(f)$ is a Reedy cofibration. Moreover, $\kappa^*_n\gamma_n^*(f)$ is a Reedy weak equivalence in $\EE^\SS$, since the objects of $\SS$ have degree $<n$. It remains to be shown that $\kappa^*_n\gamma_n^*(f)$ is a Reedy fibration whose induced maps on matching objects of degree $<n$ are trivial fibrations.

Write $\phi=\gamma_n\kappa_n$. By Lemma \ref{matching}, the functor $\phi_k^*:\EE^{\GG_k(\RR)}\to\EE^{\GG_k(\SS)}$ induces a canonical isomorphism $\phi_k^*(M_k(Y))\cong M_k(\phi^*(X))$. Therefore, the relative matching map $\phi^*(Y)\to\phi^*(X)\times_{M_k(\phi^*(X))}M_k(\phi^*(Y))$ may be identified with $\phi_k^*$ of the relative matching map $Y_k\to X_k\times_{M_k(X)}M_k(Y)$. Observe that $\phi_k^*$ preserves fibrations, thus $\kappa_n^*\gamma_n^*(g)$ is a Reedy fibration. Moreover, $M_k(\phi^*(Y))\to M_k(\phi^*(X))$ is a trivial fibration for $k<n$, since $M_k(Y)\to M_k(X)$ is so by induction hypothesis.\end{proof}

\emph{Proof of Theorem \ref{model}.} Limits and colimits in $\EE^\RR$ are constructed pointwise. The class of Reedy weak equivalences has the two-out-of-three property. Moreover, all three classes are closed under retract. It remains to be shown that the lifting and factorization axioms of a Quillen model category hold.

For the lifting axiom, observe that by Lemma \ref{LLP2}, trivial Reedy cofibrations fulfill the hypothesis of Lemma \ref{LLP1}, and therefore have the left lifting property with respect to Reedy fibrations. Dually, Lemmas \ref{RLP2} and \ref{RLP1} imply that trivial Reedy fibrations have the right lifting property with respect to Reedy cofibrations.

We now come to the factorization axiom. Given a map $f:X\to Y$ in $\EE^\RR$, we shall construct inductively a factorization $X\to A\to Y$ of $f$ into a trivial Reedy cofibration followed by a Reedy fibration.

For $n=0$, factor $f_0$ in $\EE^{\GG_0(\RR)}$ as $X_0\lrto A_0\lrto Y_0$ into a trivial cofibration followed by a fibration. Next, if $X_{\leq n-1}\to A_{\leq n-1}\to Y_{\leq n-1}$ is a factorization of $f_{\leq n-1}$ into trivial Reedy cofibration followed by Reedy fibration in $\EE^{\RR_{\leq n-1}}$, we obtain the following commutative diagram in $\EE^{\GG_n(\RR)}:$\begin{diagram}[small,silent]L_n(X)&\rTo&L_n(A)&\rTo&L_n(Y)\\\dTo&&&&\dTo\\X_n&&&&Y_n\\\dTo&&&&\dTo\\M_n(X)&\rTo&M_n(A)&\rTo&M_n(Y).\end{diagram}This diagram induces a map $X_n\cup_{L_n(X)}L_n(A)\to M_n(A)\times_{M_n(Y)}Y_n$ which we factor as a trivial cofibration followed by a fibration in $\EE^{\GG_n(\RR)}$:\begin{diagram}[small]X_n\cup_{L_n(X)}L_n(A)&\rTo^\sim&A_n&\rTo&M_n(A)\times_{M_n(Y)}Y_n.\end{diagram}The object $A_n$ of $\EE^{\GG_n(\RR)}$ together with the maps $L_n(A)\to A_n\to M_n(A)$ define an extension of $A_{\leq n-1}$ to an object $A_{\leq n}$ in $\EE^{\RR_{\leq n}}$ together with a factorization of $f_{\leq n}:X_{\leq n}\to Y_{\leq n}$ into  a Reedy cofibration $X_{\leq n}\to A_{\leq n}$ followed by a Reedy fibration $A_{\leq n}\to Y_{\leq n}$. The former map is a \emph{trivial} Reedy cofibration, because the map $X_n\to A_n$ decomposes into two maps $X_n\to X_n\cup_{L_n(X)}L_n(A)\to A_n$, the first one of which is a weak equivalence by Lemma \ref{LLP2}, the second one by construction. This defines the required factorization of $f_{\leq n}$ in $\EE^{\RR_{\leq n}}$.

The factorization of $f$ into a Reedy cofibration followed by a trivial Reedy fibration is constructed in a dual manner using Lemma \ref{RLP2} instead of Lemma \ref{LLP2}.\qed\vspace{1ex}

The proof of Theorem \ref{model} uses implicitly that trivial Reedy (co)fibrations may be characterized in terms of relative matching (latching) maps. Since this is a pivotal property of the Reedy model structure, we state it explicitly:

\begin{prp}\label{trivial}A map $f:A\to B$ in $\EE^\RR$ is a trivial Reedy cofibration if and only if, for each $n$, the relative latching map $A_n\cup_{L_n(A)}L_n(B)\to B_n$ is a trivial cofibration in $\,\EE^{\GG_n(\RR)}$.

A map $g:Y\to X$ in $\EE^\RR$ is a trivial Reedy fibration if and only if, for each $n$, the relative matching map $Y_n\to X_n\times_{M_n(X)}M_n(Y)$ is a trivial fibration in $\,\EE^{\GG_n(\RR)}$.\end{prp}

\begin{proof}For each $n$, the induced map $f_n:A_n\to B_n$ in $\EE^{\GG_n(\RR)}$ factors as$$A_n\overset{u_n}{\rto} A_n\cup_{L_n(A)}L_n(B)\overset{v_n}{\rto} B_n.$$If $f$ is a trivial Reedy cofibration then $f_n$ is a weak equivalence, so that, by Lemma \ref{LLP2}, $u_n$ is a weak equivalence, and hence $v_n$ is a trivial cofibration. Conversely, if each $v_n$ is a trivial cofibration then an induction on $n$, based on Lemma \ref{LLP2}, shows that $u_n$ is a weak equivalence, and hence $f$ is a trivial Reedy cofibration.

The dual proof for a trivial Reedy fibration $g:Y\to X$ uses Lemma \ref{RLP2} instead of Lemma \ref{LLP2}.\end{proof}

\section{Skeleta and coskeleta.}

In this section we define the skeletal filtration and the coskeletal
tower of any functor $X:\RR\to\EE$ on a generalized Reedy category
$\RR$, and study their interaction with the Reedy model structure on
$\EE^\RR$ for a Quillen model category $\EE$. We then introduce a
special class of dualizable generalized Reedy categories for which
the skeleta in $\Sets^{\RR^\op}$ are simple to describe.\vspace{1ex}

Recall that for any object $X$ of $\EE^\RR$, the restriction $j_n^*X:\GG_n(\RR)\to\EE$ along $j_n:\GG_n(\RR)\to\RR$ is denoted $X_n$. We shall write $\,t_n:\RR_{\leq n}\inc\RR$ for the full embedding of the subcategory of objects of degree $\leq n$, cf. Section \ref{groupoiddef}.

\begin{dfn} The \emph{$n$-skeleton functor} is the endofunctor $\sk_n=t_{n!}t_n^*$. The \emph{$n$-coskeleton functor} is the endofunctor $\cosk_n=t_{n*}t_n^*$.\end{dfn}

Since $t_n:\RR_{\leq n}\inc\RR$ is a full embedding, the unit of the
$(t_{n!},t_n^*)$-adjunction  (resp. the counit of the
$(t_n^*,t_{n*})$-adjunction) is an isomorphism; in other words, the
endofunctor $\sk_n$ (resp. $\cosk_n$) is an \emph{idempotent
comonad} (resp. \emph{monad}) on $\EE^\RR$.

The counit of the $(t_{n!},t_n^*)$-adjunction (resp. unit of the $(t_n^*,t_{n*})$-adjunction) provides for each object $X$ of $\EE^\RR$ a map $\sk_n(X)\to X$ (resp. $X\to\cosk_n(X)$) in $\EE^\RR$. Observe however that these maps need not be monic (resp. epic) for general $X$.

For consistency, $\sk_{-1}(X)$ (resp. $\cosk_{-1}(X)$) will denote an initial (resp. terminal) object of $\EE^\RR$.

\begin{lma}\label{latchingdef2}For each object $X$ of $\EE^\RR$, the $n$-th latching object $L_n(X)$ is canonically isomorphic to $\sk_{n-1}(X)_n$, and the $n$-th matching object $M_n(X)$ is canonically isomorphic to $\cosk_{n-1}(X)_n$.

Under these isomorphisms, the $n$-th latching map $L_n(X)\to X_n$ is induced by the counit $\sk_{n-1}(X)\to X$, and the $n$-th matching map $X_n\to M_n(X)$ is induced by the unit $X\to\cosk_{n-1}(X)$.\end{lma}

\begin{proof}This follows from the explicit formulas for the left and right Kan extensions $t_{n!}$ and $t_{n*}$ (cf. Section 3), and from axiom (iii) of Definition \ref{basic}. Indeed, the latter implies that for any object $r$ of $\RR$, the category $\RR^+(r)$ is cofinal in the comma category $\RR_{\leq n}/r$, while the category $\RR^-(r)$ is final in the comma category $r/\RR_{\leq n}$. Moreover, the latching map $L_n(X)\to X_n$ of Section \ref{latchingdef} factors canonically through the counit $\sk_{n-1}(X)_n\to X_n$, while the matching map $X_n\to M_n(X)$ of Section \ref{matchingdef} factors canonically through the unit $X_n\to\cosk_{n-1}(X)_n$.\end{proof}

\begin{lma}\label{idempotent}For any natural numbers $m\leq n$, there are canonical isomorphisms $\sk_n\circ\sk_m\cong\sk_m\cong\sk_m\circ\sk_n$ as well as $\cosk_n\circ\cosk_m\cong\cosk_m\cong\cosk_m\circ\cosk_n$.\end{lma}

\begin{proof}This follows readily from the fact that $\sk_n$ (resp. $\cosk_n$) is an idempotent comonad (resp. monad) on $\EE^\RR$.\end{proof}

Lemma \ref{idempotent} implies in particular the existence of a compatible system of maps $\sk_m\to\sk_n$ (resp. $\cosk_n\to\cosk_m$) in $\EE^\RR$. The colimit $\sk_\infty$ (resp. limit $\cosk_\infty$) of this system is isomorphic to the identity functor of $\EE^\RR$. We shall now discuss for which objects $X$ of $\EE^\RR$, this defines a skeletal filtration (resp. coskeletal tower).

Recall that a functor between Quillen model categories is called a \emph{left} (resp. \emph{right}) \emph{Quillen functor} if it preserves cofibrations and trivial cofibrations (resp. fibrations and trivial fibrations).

\begin{lma}\label{Quillen}Let $\EE$ be a Quillen model category and let $\RR$ be a generalized Reedy category. We endow $\,\EE^\RR$ and $\,\EE^{\RR_{\leq n}}$ with their Reedy model structures. Then,\begin{enumerate}\item[(i)]the left Kan extension $t_{n!}$ is a left Quillen functor;\item[(ii)]the right Kan extension $t_{n*}$ is a right Quillen functor;\item[(iii)]the restriction functor $\,t_n^*$ is simultaneously a left and right Quillen functor.\end{enumerate}In particular, $\sk_n$ (resp. $\cosk_n$) is a left (resp. right) Quillen endofunctor of $\EE^\RR$.\end{lma}

\begin{proof}By adjointness, (iii) is equivalent to the conjunction of (i) and (ii). Property (iii) follows from Proposition \ref{trivial} and Lemmas \ref{latching} and \ref{matching}.\end{proof}

\begin{prp}Let $\EE$ be a Quillen model category and let $\RR$ be a generalized Reedy category. For any $\,m<n\leq\infty$, and any object $X$ of $\,\EE^\RR$,\begin{enumerate}\item[(i)]if $X$ is Reedy cofibrant, the canonical map $\sk_m(X)\to\sk_n(X)$ is a Reedy cofibration between Reedy cofibrant objects;\item[(ii)]if $X$ is Reedy fibrant, the canonical map $\cosk_n(X)\to\cosk_m(X)$ is a Reedy fibration between Reedy fibrant objects.\end{enumerate}\end{prp}

\begin{proof}The proofs of (i) and (ii) are dual; we shall establish (i). By Lemma \ref{idempotent}, we can stick to the case $n=\infty$, i.e. to the case $\sk_m(X)\to\sk_\infty(X)=X$. For this, consider the commutative square:\begin{diagram}[small]L_k(\sk_m(X))&\rTo&L_k(X)\\\dTo&&\dTo\\\sk_m(X)_k&\rTo&X_k.\end{diagram}For $k\leq m$, the horizontal maps are isomorphisms, thus the relative latching map $\sk_m(X)_k\cup_{L_k(\sk_m(X))}L_k(X)\to X_k$ is an isomorphism too. For $k>m$, the left vertical map is an isomorphism by Lemmas \ref{latchingdef2} and \ref{idempotent}, thus the relative latching map coincides with $L_k(X)\to X_k$ which is a cofibration by hypothesis. Moreover, Lemma \ref{Quillen} shows that $\sk_m(X)$ is Reedy cofibrant.\end{proof}

We shall now introduce a special class of generalized Reedy
categories $\RR$ for which the skeletal filtration in
$\Sets^{\RR^\op}$ admits a particularly simple description, as in
Corollary \ref{skeleton} below.  In the particular case of the
simplex category $\Delta$, this proposition was first observed by
Eilenberg and Zilber (see \cite{EZ,GZ}), and therefore we have
chosen to name these special categories \emph{Eilenberg-Zilber
categories}, or briefly EZ-categories. Their formal definition is
the following:

\begin{dfn}\label{Eilenberg}An \emph{EZ-category} is a small category $\RR$, equipped with a degree-function $d:\Ob(\RR)\to\NN$, such that\begin{enumerate}\item[(i)]monomorphisms preserve (resp. raise) the degree if and only if they are invertible (resp. non-invertible);\item[(ii)]every morphism factors as a split epimorphism followed by a monomorhism;\item[(iii)]any pair of split epimorphisms with common domain has an absolute pushout.\end{enumerate}\end{dfn}

\noindent Any EZ-category is a \emph{dualizable} generalized Reedy
category where $\RR^+$ (resp. $\RR^-$) is defined to be the wide
subcategory containing all monomorphisms (resp. split epimorphisms).
Notice however that, although the dual of an EZ-category is a
generalized Reedy category, it is in general \emph{not} an
EZ-category. We are mostly interested in \emph{presheaves} on $\RR$,
so that the reader should be aware of the fact that the roles of
$\RR^\pm$ have to be reversed in the definitions of Sections 4-6.

Recall (cf. \cite{Bo}) that an \emph{absolute pushout} is a pushout
which is preserved by any functor or, equivalently, just by the
\emph{Yoneda-embedding} $\RR\inc\wRR=\Sets^{\RR^\op}$. Notice also
that any epimorphism between representable presheaves is split.
Axiom (ii) expresses thus that the epi-mono factorization system of
the presheaf topos $\wRR$ restricts (under the Yoneda-embedding) to
$\RR$, while axiom (iii) can be restated as follows: in $\,\wRR$,
the pushout of any pair of representable epimorphisms with common
domain is representable. This in turn means that for any
representable presheaf, the equivalence relation generated by two
``representable'' equivalence relations is again ``representable''.
For instance, in the simplex category $\Delta$, representable
quotients of $\Delta[m]$ correspond bijectively to ordered
partitions of $[m]$; one checks that the representable quotients
form a \emph{sublattice} of the entire quotient-lattice of
$\Delta[m]$.

The presheaf $\RR(-,r)$ represented by an object $r$ of $\RR$ will
be denoted $\RR[r]$. The split epimorphisms of an EZ-category will
be called \emph{degeneracy operators}; the monomorphisms will be
called \emph{face operators}.

Examples of EZ-categories include the simplex category $\Delta$,
Segal's category $\Gamma$, the category $\Delta_{sym}$ (see examples
\ref{basicexamples}.a-c), as well as the total category $\RR G$ of a
crossed group $G$ on a strict EZ-category $\RR$ (e.g., the category
$\Lambda$ for cyclic sets, resp. the category $\Omega$ for
dendroidal sets, see examples \ref{cyclic} and \ref{omega}). Indeed,
Proposition \ref{crossed2} shows that $\RR G$ is a dualizable
generalized Reedy category in which axiom (ii) of an EZ-category
holds; moreover, the restriction functor $\widehat{\RR G}\to\wRR$ is
monadic and hence creates absolute pushouts, so that axiom (iii) of
an EZ-category also holds.

Recall that the Yoneda-lemma allows us to identify \emph{elements}
of a set-valued presheaf $X$ on $\RR$ with \emph{maps}
$\,x:\RR[r]\to X$ in $\wRR$; such a map (or element) $x$ will be
called \emph{degenerate} if $x$ factors through a non-invertible
degeneracy $\RR[r]\to\RR[s]$, and \emph{non-degenerate} otherwise.

\begin{prp}\label{standard}Let $\,\RR$ be an EZ-category and let $X$ be a presheaf on $\RR$. Then any element $x:\RR[r]\to X$ factors in an essentially unique way as a degeneracy $\,\rho_x:\RR[r]\onto\RR[s]$ followed by a non-degenerate element $\,\sg_x:\RR[s]\to X$.\end{prp}

Any such decomposition will be referred to as a \emph{standard decomposition} of $\,x$.

\begin{proof}The existence of a standard decomposition follows from the facts that the degree-function takes values in $\NN$, and that non-invertible degeneracies lower the degree by \ref{Eilenberg}(i). For the essential uniqueness, observe first that there can be at most one comparison map from a standard decomposition $x=\sg_x\rho_x$ to another $x=\sg'_x\rho'_x$, since degeneracies are epic. It remains to be shown that such a comparison map always
exists. Take the absolute pushout of $\rho_x$ and $\rho'_x$, as
provided by \ref{Eilenberg}(iii):
\begin{diagram}[small]\RR[r]&\rTo^{\rho_x}&\RR[s]\\\dTo^{\rho'_x}&&\dTo_{\tau_x}\\\RR[s']&\rTo_{\tau'_x}&\RR[t]\end{diagram}
There exists therefore a map $\phi_x:\RR[t]\to X$ such that
$\phi_x\tau_x=\sg_x$ and $\phi_x\tau'_x=\sg'_x$. Since $\sg_x$ and
$\sg'_x$ are non-degenerate, the split epimorphisms $\tau_x$ and
$\tau'_x$ must preserve the degree. It then follows from
\ref{Eilenberg}(i) that $\tau_x$ and $\tau'_x$ are invertible so that
$(\tau'_x)^{-1}\tau_x$ provides the required comparison
map.\end{proof}


\begin{cor}\label{skeleton}Let $\,\RR$ be an EZ-category and let $X$ be a set-valued presheaf on $\RR$. Then the counit $sk_n(X)\to X$ is monic and its image is the subobject $X^{(n)}$ of those elements of $X$ which factor through an element $\RR[s]\to X$ of degree $s\leq n$.\end{cor}

\begin{proof}Notice that the counit $\sk_n(X)\to X$ factors through $X^{(n)}$ since by definition, for each object $r$ of $\,\RR$, we have$$sk_n(X)_r=t_{n!}t_n^*(X)_r=\varinjlim_{r\to t_n(s)}X_s=\varinjlim_{r\to s,d(s)\leq n}X_s.$$The induced map $\sk_n(X)\to X^{(n)}$ is pointwise surjective. It remains to be shown that $\sk_n(X)\to X^{(n)}$ is pointwise injective. Take two elements $x,y$ in $\sk_n(X)$ giving rise to the same element $z$ in $X^{(n)}$. Then, the essential uniqueness of the standard decomposition of $z$ shows that $x$ and $y$ define the same element in $\sk_n(X)$.\end{proof}

\section{Monoidal Reedy model structures}\label{solid} From now on,
we shall assume that $\EE=(\EE,\otimes_\EE,I_\EE,\tau_\EE)$ is a
\emph{closed symmetric monoidal category}, see e.g. \cite{Bo}.
Observe that if arbitrary (small) coproducts of the unit object
$I_\EE$ exist, there is a canonical functor $\Sets\to\EE$ given by
$S\mapsto\coprod_SI_\EE$. The symmetric monoidal structure will be
called \emph{solid} if these coproducts exist, and if moreover the
resulting functor from the category of sets to $\EE$ is
\emph{faithful}. Objects and morphisms of $\EE$ which belong to the
essential image of this functor will be called \emph{discrete}.
Likewise, the presheaf topos $\wRR$ maps to $\,\EE^{\RR^\op}$.
Observe that both functors have right adjoints, and hence preserve
colimits.

Recall that, according to Hovey \cite{Ho}, a \emph{monoidal model
category} is a category which is simultaneously a closed symmetric
monoidal category and a Quillen model category such that \emph{unit}
and \emph{pushout-product} axioms hold. For brevity, we shall say
that a monoidal model category $\EE$ is \emph{solid} if

\begin{enumerate}\item[(i)]the symmetric monoidal structure is
\emph{solid} in the sense mentioned above;\item[(ii)]the unit
$I_\EE$ is \emph{cofibrant};\item[(iii)]for any discrete group $G$,
discrete cofibrations in $\EE^G$ are \emph{free
$G$-extensions}\footnote{A $G$-equivariant map of $G$-sets $f:A\to
B$ is a \emph{free $G$-extension} iff $f$ is monic and $G$ acts
freely on the complement $B\backslash f(A)$}.\end{enumerate}

\noindent Observe that condition (ii) makes the unit axiom
redundant, and condition (iii) (applied to the trivial group)
implies that discrete cofibrations in $\EE$ are monic. If $\EE$ is
\emph{cofibrantly generated} (cf. \cite{Hi,Ho}) and \emph{discrete
cofibrations} in $\EE$ are \emph{monic}, then condition (iii) is
automatically satisfied, since in this case the discrete
cofibrations in $\EE^G$ are generated by free $G$-extensions.
Examples of solid monoidal model categories include the category of
compactly generated spaces, the category of simplicial sets (both
equipped with Quillen's model structure), and the category of
differential graded $R$-modules with the projective model structure.

\subsection{Boundary inclusions and cofibrations}
For each object $r$ of an EZ-category $\RR$, the \emph{formal
boundary} $\partial\RR[r]$ of $\RR[r]$ is defined to be the
subobject of those elements of $\RR[r]$ which factor through a
non-invertible face operator $s\to r$.  By Corollary \ref{skeleton},
we have $\partial\RR[r]=\sk_{d(r)-1}\RR[r]$. Our main purpose here
is to single out a class of maps in $\wRR$ which induce \emph{Reedy
cofibrations} in $\EE^{\RR^\op}$ for \emph{any} solid monoidal model
category $\EE$. This class coincides with Cisinski's class of
\emph{normal monomorphisms}, see \cite[8.1.30]{Ci}.

\begin{prp}\label{discrete}For a map $\phi:X\to Y$ of set-valued presheaves on an EZ-category $\,\RR$, the following three properties are equivalent:\begin{enumerate}\item[(i)]for each object $r$ of $\,\RR$, the relative latching map $X_r\cup_{L_r(X)}L_r(Y)\to Y_r$ is a free $\Aut(r)$-extension;\item[(ii)]$\phi$ is monic, and for each object $\,r$ of $\,\RR$ and each non-degenerate element $\,y\in Y_r\backslash\phi(X)_r$, the isotropy group $\{g\in\Aut(r)\,|\,g^*(y)=y\}$ is trivial;\item[(iii)]for each $n\geq 0$, the relative $n$-skeleton $\sk_n(\phi)=X\cup_{\sk_n(X)}\sk_n(Y)$ is obtained from the relative $(n-1)$-skeleton $\sk_{n-1}(\phi)$ by attaching a coproduct of representable presheaves of degree $n$ along their formal boundary.\end{enumerate}\end{prp}

\begin{proof}(ii)$\implies$(i). By Lemmas \ref{skeleton} and \ref{latchingdef2}, the latching object $L_r(X)$ may be identified with the subobject of \emph{degenerate} elements of $X_r$. Since $\phi$ is monic, the induced map $L_r(\phi):L_r(X)\to L_r(Y)$ is monic; moreover, since split epimorphisms have the left lifting property with respect to monomorphisms, $\phi$ takes non-degenerate elements of $X$ to non-degenerate elements of $Y$, i.e. the complement $X_r\backslash L_r(X)$ to the complement $Y_r\backslash L_r(Y)$. In particular, the relative latching map $X_r\cup_{L_r(X)}L_r(Y)\to Y_r$ is monic, and the complement of its image may be identified with the set of non-degenerate elements of $Y_r\backslash\phi(X)_r$.

  (i)$\implies$(iii). It follows from Lemmas \ref{idempotent} and \ref{latchingdef2} that the canonical map $\sk_{n-1}(\phi)\to\sk_{n}(\phi)$, evaluated at objects $r$ of degree $<n$, is an isomorphism, while at objects $r$ of degree $n$, it evaluates to $X_r\cup_{L_r(X)}L_r(Y)\to Y_r$. The latter is a free $\Aut(r)$-extension by hypothesis. Since neither $sk_{n-1}(\phi)$ nor $\sk_n(\phi)$ contain non-degenerate elements of degree $>n$, this shows that $\sk_n(\phi)$ is obtained from $\sk_{n-1}(\phi)$ by attaching, for each $\Aut(r)$-orbit in $\sk_n(\phi)_r\backslash\sk_{n-1}(\phi)_r$, a distinct copy of $\RR[r]$; since the orbit is free, the complement $\RR[r]\backslash\partial\RR[r]$ is freely attached.

(iii)$\implies$(ii). Since property (ii) is stable under pushout and sequential colimit, it suffices to show that the boundary inclusions $\partial\RR[r]\inc\RR[r]$ have property (ii). The only non-degenerate elements of $\RR[r]\backslash\partial\RR[r]$ are the automorphisms of $r$; the latter have trivial isotropy groups.\end{proof}

\noindent Cisinski shows the equivalence of \ref{discrete}(ii) and
\ref{discrete}(iii) in a slightly different setting, cf.
\cite[8.1.1, 8.1.29-35]{Ci}. In the special cases $\RR=\Gamma$ and
$\RR=\Omega$, the skeletal filtration \ref{discrete}(iii)
 has been described by Lydakis
\cite{Ly} and Moerdijk-Weiss \cite{MW2}.

A map $\phi:X\to Y$ in $\wRR$, fulfilling one of the equivalent
conditions of Proposition \ref{discrete}, will be called a
\emph{cofibration}. Condition \ref{discrete}(i) readily implies

\begin{cor}\label{discrete2}Let $\,\RR$ be an EZ-category and $\EE$ be a solid monoidal model category. A map of set-valued presheaves on $\,\RR$ is a cofibration if and only if the induced map in $\,\EE^{\RR^\op}$ is a Reedy cofibration.\end{cor}


\begin{dfn}\label{monoidal}An EZ-category $\,\RR$ is called \emph{quasi-monoidal}
if the presheaf topos $\wRR$ carries a symmetric monoidal structure
$(\wRR,\square,I_\square,\tau_\square)$ such
that\begin{enumerate}\item[(i)]the bifunctor
$-\square-:\wRR\times\wRR\to\wRR$ preserves colimits in both
variables;\item[(ii)]the unit $I_\square$ is
cofibrant;\item[(iii)]for all objects $\,r,s$ of $\,\RR$, the
boundary inclusions induce a pullback
square\begin{diagram}[small]\partial\RR[r]\square\partial\RR[s]&\rTo&\partial\RR[r]\square\RR[s]\\\dTo&&\dTo\\\RR[r]\square\partial\RR[s]&\rTo&\RR[r]\square\RR[s]\end{diagram}
in $\wRR$ consisting of cofibrations.
\end{enumerate}\end{dfn}

\noindent Since cofibrations are monic, the induced map
$$\RR[r]\square\partial\RR[s]\cup_{\partial\RR[r]\square\partial\RR[s]}\partial\RR[r]\square\RR[s]\to\RR[r]\square\RR[s]$$is
also monic, and hence, by \ref{discrete}(ii), a cofibration. It then
follows from \ref{discrete}(iii) and \ref{monoidal}(i) that the
class of cofibrations in $\wRR$ satisfies Hovey's pushout-product
axiom.

\begin{thm}\label{monoidalmodel}Let $\RR$ be a quasi-monoidal EZ-category and let $\EE$
be a cofibrantly generated, solid monoidal model category. Then the
functor category $\,\EE^{\RR^\op}$, equipped with the Reedy model
structure of \ref{model} and with the symmetric monoidal structure
obtained by convolution, is a cofibrantly generated, solid monoidal
model category.\end{thm}

\begin{proof}We shall first show that $\EE^{\RR^\op}$ is cofibrantly
generated, then define the symmetric monoidal structure
$\square_\EE$ on $\EE^{\RR^\op}$, and finally check the
pushout-product axiom for the generating (trivial) cofibrations of
$\EE^{\RR^\op}$.

The generating (trivial) Reedy cofibrations of $\EE^{\RR^\op}$ are
obtained by ``twisting'' the generating (trivial) cofibrations of
$\EE$ against the boundary inclusions of $\wRR$. To be more precise,
let $f:A\to B$ be an arbitrary generating (trivial) cofibration of
$\EE$ and let $i_r:\partial\RR[r]\to\RR[r]$ be a boundary inclusion
of $\wRR$. For brevity, for any object $A$ of $\EE$ and any set $S$,
the tensor $A\otimes_\EE(\coprod_SI_\EE)$ will be written
$A\otimes_\EE S$, and similarly for set-valued presheaves on $\RR$.
We thus obtain the following commutative
square\begin{diagram}[small]A\otimes_\EE\partial\RR[r]&\rTo&B\otimes_\EE\partial\RR[r]\\\dTo&&\dTo\\A\otimes_\EE\RR[r]&\rTo&B\otimes_\EE\RR[r]\end{diagram}
in $\EE^{\RR^\op}$. The induced comparison
map$$A\otimes_\EE\RR[r]\cup_{A\otimes_\EE\partial\RR[r]}B\otimes_\EE\partial\RR[r]\to
B\otimes_\EE\RR[r]$$ is a \emph{generating (trivial) Reedy
cofibration} of $\EE^{\RR^\op}$, and they are all of this form.
Indeed, since the Reedy model structure on $\EE^{\RR^\op}$ is well
defined by Theorem \ref{model}, the generating property just
expresses that a map $X\to Y$ is a trivial Reedy fibration (resp.
Reedy fibration) \emph{if and only if} it has the right lifting
property with respect to the generating Reedy cofibrations (resp.
trivial Reedy cofibrations). This in turn follows from the fact
that, by adjointness, one of the following two squares
\begin{diagram}[small]A\otimes_\EE\RR[r]\cup_{A\otimes_\EE\partial\RR[r]}B\otimes_\EE\partial\RR[r]&\rTo&X&\quad\quad&A&\rTo&X_r\\\dTo&&\dTo&\quad\quad&\dTo&&\dTo\\B\otimes_\EE\RR[r]&\rTo&Y&\quad\quad&B&\rTo&Y_r\times_{M_r(Y)}M_r(X)\end{diagram} has a diagonal filler if and only if the other has.

The symmetric monoidal structure
$-\square_\EE-:\EE^{\RR^\op}\times\EE^{\RR^\op}\to\EE^{\RR^\op}$ is
defined by the universal property that for any objects $\,X,Y,Z$ of
$\EE^{\RR^\op}$, maps $X\square_\EE Y\to Z$ in $\EE^{\RR^\op}$
correspond bijectively to natural systems of maps $X_r\otimes_\EE
Y_s\to Z_t$ in $\EE$, indexed by maps $\RR[t]\to\RR[r]\square\RR[s]$
in $\wRR$. In other words, we have the
formula\begin{gather*}(X\square_\EE
Y)_t=\varinjlim_{\RR[t]\to\RR[r]\square\RR[s]}X_r\otimes_\EE
Y_s.\end{gather*}In particular, the monoidal structure on
$\EE^{\RR^\op}$ is closed (cf. \cite[appendix]{MW}) and extends the
given one on $\wRR$; both share the same unit $I_\square$ (which is
cofibrant by \ref{monoidal}(ii) and \ref{discrete2}) so that the
canonical map $\Sets\to\EE^{\RR^\op}$ factors through $\,\wRR$.
Therefore, all axioms of a cofibrantly generated, solid monoidal
model category are satisfied, except possibly the pushout-product
axiom. In order to establish the latter, take two generating
cofibrations $f:A\to B,\,g:C\to D$ in $\EE$ as well as two boundary
inclusions $\,i_r,i_s$ of $\,\wRR$, and consider the associated
generating Reedy cofibrations
\begin{align*}f/i_r&:A\otimes_\EE\RR[r]\cup_{A\otimes_\EE\partial\RR[r]}B\otimes_\EE\partial\RR[r]\to
B\otimes_\EE\RR[r],\\
g/i_s&:C\otimes_\EE\RR[s]\cup_{A\otimes_\EE\partial\RR[s]}D\otimes_\EE\partial\RR[s]\to
D\otimes_\EE\RR[s].\end{align*}We shall denote them by
$f/i_r:A/i_r\to B/i_r$ and $g/i_s:C/i_s\to D/i_s$. We have to show
that the pushout-product
map\begin{gather}\label{pp}\left(A/i_r\square_\EE
D/i_s\right)\cup_{\left(A/i_r\square_\EE
C/i_s\right)}\left(B/i_r\square_\EE C/i_s\right)\to
\left(B/i_r\square_\EE D/i_s\right)\end{gather} is a Reedy
cofibration which is trivial if $f$ or $g$ is trivial.

The operation $(f,i_r)\mapsto f/i_r$ extends in an evident way to a
bifunctor
$$-/-:\Arr(\EE)\times\Arr(\wRR)\to\Arr(\EE^{\RR^\op}),$$where $\Arr(\CC)$ denotes the category of arrows in $\CC$.

It is now straightforward to verify that (\ref{pp}) is isomorphic to
$h/i_t$ where \begin{align*}h:A\otimes_\EE D&\cup_{A\otimes_\EE
C}B\otimes_\EE C\to B\otimes_\EE D\quad\text{and}\\
i_t:\RR[r]\square\partial\RR[s]&\cup_{\partial\RR[r]\square\partial\RR[s]}\partial\RR[r]\square\RR[s]\to\RR[r]\square\RR[s]\end{align*}
are the canonical comparison maps. Since in $\EE$ and $\wRR$ the
pushout-product axiom holds, it remains to be shown that for a
(trivial) cofibration $h$ in $\EE$, and cofibration $i_t$ in $\wRR$,
the map $h/i_t$ is a (trivial) Reedy cofibration in $\EE^{\RR^\op}$.
By the adjointness argument given above, this holds whenever $h$ is
a generating (trivial) cofibration, and $i_t$ a boundary inclusion;
the general case reduces to this special case, since the operation
$(h,i_t)\mapsto h/i_t$ commutes with sequential colimits and
retracts in each variable, and takes pushout squares in each
variable to pushout squares in $\EE^{\RR^\op}$.\end{proof}

\begin{exms}\label{monoidalEil}(a) The simplex category $\Delta$ is a quasi-monoidal EZ-category for
the cartesian product on $\widehat{\Delta}$. Therefore, the category
of simplicial spaces is a monoidal model category for the cartesian
product, where ``space'' means either compactly generated
topological space or simplicial set. In the latter case, the Reedy
cofibrations are precisely the monomorphisms, and the result is of
course well-known. Notice that in general, even for strict
EZ-categories $\RR$, the exactness axiom \ref{monoidal}(iii) may not
be true for the cartesian product on $\wRR$.\vspace{1ex}

(b) Segal's \cite{Se} category $\Gamma$ is a (quasi-)monoidal
EZ-category for the smash product on $\widehat{\Gamma}$, as can be
deduced from the work of Lydakis \cite{Ly}. This means that the
category of $\Gamma$-spaces, equipped with the strict model
structure of Bousfield-Friedlander \cite{BF}, is a monoidal model
category. \vspace{1ex}

(c) The category $\Omega$ for dendroidal sets (see \ref{omega}) is a
quasi-monoidal EZ-category for the \emph{Boardman-Vogt tensor
product} on $\widehat{\Omega}$, cf. \cite{MW,CM}. Therefore, the
category of dendroidal spaces is a monoidal model category in such a
way that the embedding $i:\Delta\inc\Omega$ induces a monoidal
Quillen adjunction between simplicial spaces and dendroidal spaces.
It can be shown that, in complete analogy to Rezk's localization of
simplicial spaces (the model structure for complete Segal spaces,
cf. \cite{Bg, JT, Lu, Rez}), there is a localization of the model
category of dendroidal spaces which is Quillen equivalent to the
model category of quasi-operads introduced in
\cite{CM}.\end{exms}\vspace{1ex}

\noindent{\small\sc Universit\'e de Nice, Lab. J.-A. Dieudonn\'e,
Parc Valrose, 06108 Nice, France.}\hspace{2em}\emph{E-mail:}
cberger$@$math.unice.fr

\noindent{\small\sc Mathematisch Instituut, Postbus 80.010, 3508 TA
Utrecht, The Netherlands.}\hspace{2em}\emph{E-mail:}
moerdijk$@$math.uu.nl

\end{document}